\documentclass[12pt,a4paper,notitlepage]{article}
\usepackage{amsfonts}
\usepackage{amsmath}

\input{tcilatex}

\begin{document}

\title{Diffusive limits on the Penrose tiling}
\author{A. Telcs \\
Magyar tud\'{o}sok k\"{o}r\'{u}tja 2, Budapest, Hungary, H-1117}
\maketitle

\begin{abstract}
In this paper random walks on the Penrose lattice are investigated. \ Heat
kernel estimates and the invariance principle are shown.
\end{abstract}


\section{Introduction}

The Penrose tiling \cite{GS,Sen} is the most famous nonperiodic tiling of
the plane and is an unfailing source of beautiful properties and phenomena
to be explored \cite{A}. There is associated with the Penrose tiling by the
usual duality relation a lattice that may be called the Penrose lattice.
Kunz \cite{K} discussed conditions under which simple nearest-neighbor
random walk on the Penrose lattice would conform to a version of the central
limit theorem, but did not completely resolve the issue. In the present
paper, it is proved that random walk on the Penrose lattice satisfies the
invariance principle (in annealed sense, see detailed explanation in the
next section) that with appropriate scalings, the random walk process
converges weakly to a non-degenerate rotation-invariant Brownian motion. Sz%
\'{a}sz \cite{Sz} has conjectured that a related Lorentz scatter system also
satisfies this principle. The proof of the discrete, random walk counterpart
of the conjecture, also formulated by Sz\'{a}sz, is the main result of the
paper.

Central to the approach followed here is the notion of roughly isometric
weighted graphs, defined precisely later, which share diffusion properties
(cf. \cite{HK}). This enables us to relate the simple random walk on the
Penrose lattice to translation invariant walks on the plane square lattice $%
\mathbb{Z}^{2}$. Let $d(x,y)$ denote the graphical distance (number of edges
in the shortest path) between vertices $x$ and $y$. Many random walk
processes obey two-sided Gaussian heat kernel estimate $\left( GE_{\alpha
,2}\right) $%
\begin{equation}
\frac{c}{n^{\alpha /2}}\exp \left( -\frac{Cd(x,y)^{2}}{n}\right) \leq 
\widetilde{p}_{n}\left( x,y\right) \leq \frac{C}{n^{\alpha /2}}\exp \left( -%
\frac{cd(x,y)^{2}}{n}\right)  \label{GEaa2}
\end{equation}%
for $0<n<d(x,y)$, with $\alpha \geq 1$, $c>0$ and $C>0$. Here $\widetilde{%
p_{n}}=p_{n}+p_{n+1}$, where $p_{n}(x,y)$ is the probability that a walker
departing from vertex $x$ is found at vertex $y$ after $n$ steps. Delmotte 
\cite{D} has shown that the bounds $\left( GE_{2}\right) $: 
\begin{equation}
\frac{c}{V\left( x,\sqrt{n}\right) }\exp \left( -\frac{Cd(x,y)^{2}}{n}%
\right) \leq \widetilde{p}_{n}\left( x,y\right) \leq \frac{C}{V\left( x,%
\sqrt{n}\right) }\exp \left( -\frac{cd(x,y)^{2}}{n}\right) ,  \label{GE2}
\end{equation}%
(where $V\left( x,r\right) $ denotes the volume of geodesic balls in graph
distance) are stable under rough isometry for random walks on weighted
graphs (which we define precisely later). In other words if two graphs are
roughly isometric and $\left( \ref{GE2}\right) $ holds for one then it holds
for the other as well. \ We know that $\left( \ref{GE2}\right) $ holds for
the simple symmetric random walk on $\mathbb{Z}^{2}$ ( $V\left( x,r\right)
\simeq r^{2},$ i.e. satisfy $\left( \ref{GEaa2}\right) $ with $\alpha =2$)
and consequently holds for the Penrose graph if it that graph is roughly
isometric to $\mathbb{Z}^{2}$. \ A very short, direct proof will be given of
rough isometry between the Penrose graph and $\mathbb{Z}^{2}$. \ 

Unfortunately the exponents in the upper and lower estimates in $\left( \ref%
{GEaa2}\text{ and }\ref{GE2}\right) $ contain different constants which
reflects local inhomogeneities. \ The rough isometry invariance and the
Gaussian estimate $\left( \ref{GE2}\right) $ have been proved along a series
of estimates in which some cumulation of constants is unavoidable.
Consequently if we are looking for the central limit theorem or for the
invariance principle we need a different approach.

In the field of stochastic processes and statistical physics a powerful
method has been developed to investigate random walks in random environment,
on percolation clusters and interacting particle systems. \ A key result in
this direction is the celebrated paper by Kipnis and Varadhan \cite{KV} and
its influential extension by De Masi, Ferrari, Goldstein, Wick \cite{MFGW}
in which annealed central limit theorem and invariance principle is shown,
(the initial environment is averaged with respect to an invariant measure),
(for further details see Section \ref{s3} and \cite{SS} on convergence
notions in random environments). \ That result provides us immediate
derivation of the central limit theorem and the invariance principle for the
random walk on the Penrose lattice in the same sense.

In what follows we introduce the basic terminology then the statement is
proved.

\section{Preliminaries\label{s1}}

We will consider infinite connected graphs with vertex set $\Gamma $. If an
edge joins vertices $x$ and $y$ we denote that edge by $x\sim y$. The
distance $d\left( x,y\right) $ will be the shortest path metric. In
particular we will speak about the integer lattice $\mathbb{Z}^{m}=\left( 
\mathbb{Z}^{m},d\right) $ graph where vertexes are elements of $\mathbb{Z}%
^{m}$ and $x,y\in \mathbb{Z}^{m}$ form an edge, $x\sim y$, if and only if $%
\left\vert x-y\right\vert =1$. We will speak about the integer lattice $%
\left( \mathbb{Z}^{m},\left\vert .\right\vert \right) $ if we consider the
same vertex and edge set but the metric is the Euclidean one. \ We do not
define the Penrose tiling ( cf. \cite{dB}), we assume that it is well
defined and given for us on $\mathbb{R}^{2}$.

The Penrose lattice $\left( \Gamma ,\left\vert .\right\vert \right) $ is a
metric space. It is the set of\emph{\ centers }(centroids)\emph{\ }of the
tiles equipped with the Euclidean distance. Two tiles are neighbors if they
are edge adjacent. Two vertexes of the Penrose lattice are neighbors if they
centers of neighboring tiles. \ Those vertexes form edges of the lattice. \
We will speak about Penrose graph, with the same vertex and edge set but
with $d\left( x,y\right) $, the shortest path graphs distance. Let $\Gamma
=\left( \Gamma ,d_{P}\right) $ denote the Penrose graph, and $\left( \mathbb{%
Z}^{2},d_{\mathbb{Z}}\right) $ integer lattice graph.

We distinguish tilings by fixing a reference vertex and identifying it with
the origin of $\mathbb{R}^{2}.$ \ We denote by $\Omega $ the space of
tilings (union of ten tori $\Omega _{i.,j}\subset \mathbb{R}^{2}$ $i,j=0..4,$
$i<j$ , for details cf. Section 2.1 in \cite{K}) . \ Let $d\left( x\right) $
be the degree of $x$, the number of neighbors.

\begin{definition}
A \ graph is weighted if a symmetric weight function $\mu _{x,y}>0$ is given
on the edges. This weight defines a measure on vertexes and sets:%
\begin{eqnarray*}
\mu \left( x\right) &=&\sum_{y\sim x}\mu _{x,y} \\
\mu \left( A\right) &=&\sum_{x\in A}\mu \left( x\right)
\end{eqnarray*}%
We denote the ball of radius of $r$ by $B\left( x,r\right) =\left\{
y:d\left( x,y\right) <r\right\} $ and we call $V\left( x,r\right) =\mu
\left( B\left( x,r\right) \right) $ its volume. \ In particular for the
Penrose lattice (and graph) $\mu _{x,y}\equiv 1$ if $x\sim y$ is an edge,
while $\mu _{x,y}$ is zero otherwise. The same applies for the integer
lattice.
\end{definition}

A Markov chain, with transition probabilities $P\left( x,y\right) $ is
reversible ($\mu $-reversible) if there is a $\mu $ measure such that $\mu
\left( x\right) P\left( x,y\right) =\mu \left( y\right) P\left( y,x\right) $.

\begin{definition}
In general a random walk\ $X_{n}$ on $\Gamma $, a weighted graph, with $\mu $
is a reversible Markov chain defined by the one step transition
probabilities:%
\begin{equation*}
P\left( X\,_{n}=y|X_{n-1}=x\right) =P\left( x,y\right) =\frac{\mu _{x,y}}{%
\mu \left( x\right) }.
\end{equation*}
\end{definition}

The random walk on the Penrose lattice (and graph) is reversible Markov
chain with transition probability $P\left( x,y\right) =1/d\left( x\right)
=1/4$ \ for $x\sim y$. It is clear that $d\left( x\right) P\left( x,y\right)
=d\left( y\right) P\left( y,x\right) =1$.\ Denote $X_{i}$ the actual
position of the Markov chain (random walk) which is well-defined for any
fixed $X_{0}\in \Gamma $ and it is the reference vertex of the Penrose
lattice.

\section{Heat kernel estimate for the Penrose graph\label{s2}}

First of all we give the definition the bi-Lipschitz property and rough
isometry.

\begin{definition}
A metric space $\left( \Gamma ,d\right) $ is bi-Lipschitz to $\left( \Gamma
^{\prime },d^{\prime }\right) $ if there is a bijection $\Phi $ from $\Gamma 
$ to $\Gamma ^{\prime }$ and a constant $C>1$ such that for all $x\neq y\in
\Gamma $%
\begin{equation}
\frac{1}{C}d\left( x,y\right) \leq d^{\prime }\left( \Phi \left( x\right)
,\Phi \left( y\right) \right) \leq Cd\left( x,y\right)  \label{bl}
\end{equation}
\end{definition}

\begin{definition}
Two weighted graphs $\Gamma $ with $\mu $ and $\Gamma ^{\prime }$ with $\mu
^{\prime }$ are roughly isometric (or quasi isometric) (cf. \cite[Definition
5.9]{BK}) if there is a map $\phi $ from $\Gamma $ to $\Gamma ^{\prime }$
such that there are $a,b,c,M>0$ for which
\end{definition}

\begin{equation}
\frac{1}{a}d\left( x,y\right) -b\leq d^{\prime }\left( \phi \left( x\right)
,\phi \left( y\right) \right) \leq ad\left( x,y\right) +b  \label{r1}
\end{equation}%
for all $x,y\in \Gamma ,$

\begin{equation}
d^{\prime }\left( \phi \left( \Gamma \right) ,y^{\prime }\right) \leq M
\label{r2}
\end{equation}%
for all $y^{\prime }\in \Gamma ^{\prime }$ and 
\begin{equation}
\frac{1}{c}\mu \left( x\right) \leq \mu ^{\prime }\left( \phi \left(
x\right) \right) \leq c\mu \left( x\right)  \label{r3}
\end{equation}%
for all $x\in \Gamma .$

\begin{remark}
\label{r4}It is clear that if $\phi $ from $\Gamma $ to $\Gamma ^{\prime }$
is a rough isometry then there is a rough isometry $\phi ^{\prime }$ from $%
\Gamma ^{\prime }$ to $\Gamma $ as well..
\end{remark}

\begin{theorem}
\label{p1}(Solomon \cite{S}) The Penrose lattice is bi-Lipschitz to the
integer lattice.
\end{theorem}

\begin{proposition}
\label{p2}The Penrose lattice is rough isometric to the integer lattice.
\end{proposition}

The statement follows from the bi-Lipschitz property. \ A very short and
direct proof can be given, which we present here.

\begin{proof}[Proof of Proposition \protect\ref{p2}]
Denote by $m$ the smaller distance between the opposite boundaries of the
thin rhombus and write $\varepsilon =\sqrt{2}m/4$. Consider the integer
lattice $\varepsilon \mathbb{Z}^{2}$. It is clear that if an open square
with edge length $\varepsilon $ contains a center of a rhombus that it is
fully contained by the closed rhombus. \ Let $\Psi $ map the center of the
rhombus to the center of the square. \ It is clear that $\Psi $ is rough
isometry from the Penrose lattice to the integer lattice .
\end{proof}

\begin{proposition}
\label{p3}The Penrose graph is roughly isometric to the integer lattice
graph.
\end{proposition}

\begin{proof}
Let us consider $\Psi ,$ the map introduced above, between $\Gamma $ and $%
\varepsilon \mathbb{Z}^{2}$. Now we consider the graph distances $d_{P},d_{%
\mathbb{Z}}$. \ \ It is clear that 
\begin{equation*}
d_{P}\left( x,y\right) \leq d_{\mathbb{Z}}\left( \Psi \left( x\right) ,\Psi
\left( y\right) \right) .
\end{equation*}%
The opposite inequality is also easy. \ Let $2L$ be the maximal number of
squares which is needed to cover the largest diagonal of rhombi. It is clear
that $L$ is bounded since the diameter of the rhombi is also bounded. \ Then 
\begin{equation*}
d_{\mathbb{Z}}\left( \Psi \left( x\right) ,\Psi \left( y\right) \right) \leq
2Ld_{P}\left( x,y\right) .
\end{equation*}%
It is also clear that the conditions $\left( \ref{r2},\ref{r3}\right) $ are
satisfied.
\end{proof}

\begin{lemma}
\label{L1} The Penrose lattice and the Penrose graph are roughly isometric.
\end{lemma}

\begin{proof}
Let $\Phi _{1}$ the rough isometry from the graph $\Gamma $ to the lattice $%
Z^{2}$, let $\Phi _{2}$ the identity map on $Z^{2}$ which is bi-Lipschitz
between the integer lattice and graph, finally $\Phi _{3}$ the rough
isometry from the integer graph to the Penrose graph. The existence of $\Phi
_{3}$ follows from Proposition \ref{p3} and Remark \ref{r4}. Then \ $\Phi
=\Phi _{3}\circ \Phi _{2}\circ \Phi _{1}$ is rough isometry between the
Penrose lattice and graph.
\end{proof}

Now we recall Delmotte's result \cite{D} omitting the third equivalent
statement, the parabolic Harnack inequality, since we do not need it in the
sequel.

\begin{theorem}
\label{De}Let $\Gamma $ with $\mu $ be a weighted graph. Assume that there
is a \ $p_{0}>0$ such that for all edges $P\left( x,y\right) \geq p_{0}$.
Then the following statements are equivalent.\newline
\newline
1. there are $C,c>0,$ $\alpha \geq 1$ such that for all $x,y\in \Gamma $ and 
$n>0$%
\begin{equation}
\frac{c}{V\left( x,\sqrt{n}\right) }\exp \left( -C\frac{d\left( x,y\right)
^{2}}{n}\right) \leq \widetilde{p}_{n}\left( x,y\right) \leq \frac{C}{%
V\left( x,\sqrt{n}\right) }\exp \left( -c\frac{d\left( x,y\right) ^{2}}{n}%
\right)  \label{Ge22}
\end{equation}%
holds,\newline
2.\newline
$\left( i\right) $The volume doubling condition $\left( VD\right) $ holds:
there is a $C>0$ such that for all $x\in \Gamma ,$ $r\geq 1$%
\begin{equation*}
V\left( x,2r\right) \leq CV\left( x,r\right)
\end{equation*}%
and \newline
$\left( ii\right) $ the Poincare inequality $\left( PI_{2}\right) $ holds:%
\newline
there is a $C>0,$ such that for all $x$ and, $r>1$,$f:B\left( x,r\right)
\rightarrow \mathbb{R}$ 
\begin{equation*}
\sum_{y\in B\left( x,r\right) }\left( f\left( y\right) -f_{B}\right) ^{2}\mu
\left( y\right) \leq cr^{2}\sum_{y,z\in B\left( x,r\right) }\left( f\left(
y\right) -f\left( z\right) \right) ^{2}\mu _{y,z}
\end{equation*}%
where $f_{B}=\frac{1}{V\left( x,r\right) }\sum_{y\in B}f\left( y\right) \mu
\left( y\right) ,$ $f\neq 0,$ $B=B\left( x,r\right) $.
\end{theorem}

It is well known that the volume doubling property as well as the Poincare
inequality are rough isometry invariant. \ Of course volume doubling can be
replaced with $V\left( x,r\right) \simeq r^{\alpha }$ and $\left( \ref{Ge22}%
\right) $ reduces into $\left( \ref{GEaa2}\right) .$

\begin{corollary}
\label{CGE2} If $\Gamma $ with $\mu $ and $\Gamma ^{\prime }$ with $\mu
^{\prime }$ are roughly isometric graphs then $\left( GE_{\alpha ,2}\right) $
holds for one if and only if holds for the other.
\end{corollary}

\begin{theorem}
The Gaussian estimate $\left( GE_{2,2}\right) $ holds for the random walk on
the Penrose graph.
\end{theorem}

\begin{proof}
Proposition \ref{p2} ensures that Penrose graph is rough isometric to $%
\mathbb{Z}^{2}$. \ It is well-known that $\left( GE_{2,2}\right) $ holds for
the random walk on the integer lattice (graph), and then by Corollary \ref%
{CGE2} $\left( GE_{2,2}\right) $ holds for the random walk on the Penrose
graph as well.
\end{proof}

\section{The invariance principle\label{s3}}

In this section we confine ourself to the Penrose lattice. The Gaussian
estimate $\left( GE_{2,2}\right) $ provides a nice description of the random
walk on the Penrose graph but the different constants in the exponents mean
that we have only estimate of the variance and it may change from place to
place as well as in time. \ Particularly we do not know if the properly
scaled mean square displacement $\frac{1}{n}E\left( d^{2}\left(
X_{0},X_{n}\right) \right) $ has a limit. \ Of course we expect that due to
the asymptotic spherical symmetry of the Penrose tiling the diffusion matrix
is the identity matrix up to a fixed constant multiplier. In other words the
scaled mean square displacement is direction independent. In order to obtain
the invariance principle for the Penrose lattice we need a different method.
\ This is the method of ergodic processes of the environment "seen from the
tagged particle" (cf. \cite{KV},\cite{MFGW}). Thanks to the result of De
Masi \& all \cite{MFGW} it is enough to check that the conditions of Theorem
2.1 in \cite{MFGW} are satisfied and that the covariance matrix is positive
definite.

There are several formulation of the invariance principle, (see for the
classical formulation in \cite{B}). We say that $X_{t}$ satisfies the
central limit theorem $\left( CLT\right) $ if there is a $\sigma \geq 0$
such that%
\begin{equation*}
\frac{X_{t}}{\sqrt{t}}\rightarrow N\left( 0,\sigma \right)
\end{equation*}%
in distribution. \ What is slightly stronger, the process can be re-scaled,
that is for $\varepsilon \rightarrow 0$ for all $t$ 
\begin{equation}
X^{\left( \varepsilon \right) }=\varepsilon X_{t/\varepsilon
^{2}}\rightarrow W_{\sigma }\left( t\right)  \label{rs}
\end{equation}%
where Wiener process with variance $\sigma t$ in the sense of finite
dimensional distributions. The invariance principle holds if the convergence
holds for the path-space measures for the processes (cf. Goldstein \cite{G}%
). \ This requirement is equivalent with $\left( \ref{rs}\right) $ and
"tightness" of the process (cf. \cite{B}).

Let $\mathbb{P}_{\mu }$ the path-space measure of the processes $\left(
X_{t}^{\left( \varepsilon \right) }\right) _{t\geq 0},\varepsilon >0.$
Following \cite{MFGW} we say that $X^{\left( \varepsilon \right) }=\left(
X_{t}^{\left( \varepsilon \right) }\right) _{t\geq 0}$ converges weakly in $%
\mu $-measure to $Y$ if for any continuous function $F$ on the path-space $%
D\left( \left[ 0,\infty \right) ,\Gamma \right) $:%
\begin{equation*}
\mathbb{E}_{\mu }\left( F\left( X^{\left( \varepsilon \right) }\right)
\right) \rightarrow \mathbb{E}_{\mu }\left( Y\right) .
\end{equation*}

\begin{definition}
\label{DIP}We say that the invariance principle holds if the weak
convergence in $\mu $-measure to the Wiener process holds in this sense.
\end{definition}

We consider the environment process $\omega _{n}$ seen from the particle. \
It is more convenient to use (as it is done by Kunz in \cite{K}) the Markov
chain $z_{n}=\left( \omega _{n},X_{n}\right) $. \ Let us note that $\mathbb{E%
}_{\mu }\left( F\left( X^{\left( \varepsilon \right) }\right) \right) $
means that the underlying environment process is started from the invariant
measure $\mu $, $\mathbb{E}_{\mu }\left( F\left( X^{\left( \varepsilon
\right) }\right) \right) =\mathbb{E}_{\mu }\left( F\left( X^{\left(
\varepsilon \right) }\right) |\omega _{0}=\omega \right) $ $\omega $ is
chosen according to $\mu $, in other words we average with respect to the
(initial) environment and obtain an annealed type of result (cf. \cite{SS}).

\begin{theorem}
\label{T1}The random walk on the Penrose lattice satisfies the invariance
principle with non-degenerate covariance matrix.
\end{theorem}

\begin{proof}
\ Kunz has shown that $z_{n}$ is ergodic (see also a more general result by
Robinson \cite{R}). The invariant measure is combination of the Lebesgue
measure $\lambda $ on the tori $\Omega _{i,j}$. \ \ For $0\leq i<j\leq 4$%
\begin{equation*}
\mu |_{\Omega _{i,j}}=\tau ^{1-\left\lfloor \frac{j-i}{2}\right\rfloor
}\lambda
\end{equation*}%
where $\tau $ is the "Golden mean" $\left( \sqrt{5}+1\right) /2$. \ It is
clear that $X_{n}=\sum_{i=1}^{n}V\left( z_{i-1},z_{i}\right) $ where $%
V\left( z_{i-1},z_{i}\right) =X_{i}-X_{i-1}$ is an antisymmetric function,
(cf. \cite{MFGW} (2.3),(2.6) and the remark below it.) \ It follows from the
definition that $X_{n}$ and $z_{n}$ as well are reversible. The random walk
is well defined for all $\omega $, for all Penrose lattice, with given
reference vertex, hence we have the path metric $\mathbb{P}_{\omega }$.
Similarly $\mathbb{P}_{\mu }$ is well defined if the initial lattice is
chosen according to the invariant measure. \ Let $\mathbb{E}_{\mu }$ denote
the corresponding expected value.\ The only properties are to check that the
conditional drift 
\begin{equation}
\varphi =\mathbb{E}_{\mu }\left( X_{1}-X_{0}|X_{0}\right)  \label{fi}
\end{equation}%
exists and that the covariance matrix%
\begin{equation}
D=\mathbb{E}_{\mu }\left( \left( X_{1}-\varphi \right) \left( X_{1}-\varphi
\right) ^{\ast }\right)  \label{d}
\end{equation}%
is non-degenerate. \ For any given $X_{0}=x$ the conditional drift evidently
exists thanks to the bounded distances of neighbors. \ The conditions of the
main result of \cite{MFGW} are satisfied, hence the invariance principle
holds for $X_{i}\ $\ in the sense of Definition \ref{DIP}.

Let us recall that the $D$ always exists (see Remark 1. below \ (2.30) in 
\cite{MFGW}).

We show that the covariance matrix is positive definite. Let us consider the
annulus $B\left( 0,C_{2}\sqrt{n}\right) \backslash B\left( 0,C_{1}\sqrt{n}%
\right) $ intersected with the cone about a given direction $e\in \mathbb{R}%
^{2}$ with angle\ $\alpha $ fixed $\pi /2>\alpha >0.$ \ Let $H$ denote the
intersection. The constants $C_{1},C_{2}$ are arbitrary and fixed. \ Let us
recall (Lemma \ref{L1}) that the Penrose lattice and graph are roughly
isometric 
\begin{eqnarray*}
&&\frac{1}{n}\mathbb{E}\left( e^{\ast }X_{n}X_{n}^{\ast
}Ee|X_{_{0}}=x_{0}\right) \\
&=&\frac{1}{n}\mathbb{E}\left( \left( e^{\ast }X_{n}\right)
^{2}|X_{_{0}}=x_{0}\right) \\
&\geq &\frac{1}{n}\sum_{x\in H}\left( ex\right) ^{2}P_{n}\left(
x_{0},x\right) \\
&\geq &\frac{\left\vert H\right\vert }{n}c^{\prime }\left( \cos \left(
\alpha \right) \sqrt{n}\right) ^{2}\frac{c^{\prime \prime }}{n}\exp \left[ -C%
\frac{\left( aC_{2}\sqrt{n}+b\right) ^{2}}{n}\right] \geq c>0
\end{eqnarray*}%
independently of $x_{0}$ and $e$, hence the covariance matrix is
non-degenerate. (Here the effect of all previous constants are absorbed into
the last constant $c$.) By this we have shown that the invariance principle
holds for the random walk on almost all Penrose lattice and the limiting
process is a non-degenerate Brownian motion.
\end{proof}

\begin{remark}
The results presented in this paper carry over easily to other quasicrystals
which can be constructed by the projection methods similar to the one
produces the Penrose tiling. This applies to generalized Penrose tilings
(produced by $p$-grids), higher dimensional Penrose tilings and stochastic
tilings.
\end{remark}

\begin{remark}
It seems plausible that with some extra work one can show that the
covariance matrix is the identity matrix multiplied with a positive
constant. The exact value of the constant ought to be determined as well.
\end{remark}

\section{Acknowledgement}

The author expresses his sincere thanks to Domokos Sz\'{a}sz for the
inspiring question and M\'{a}rton Bal\'{a}zs for useful discussions. Thanks
are due to P\'{e}ter N\'{a}ndori for useful comments on the draft of the
paper.

Particular thanks are due to the referees for their helpful, detailed
comments and for the invested work which is far exceeding the usual
contribution.

\end{document}